\input ./style/arxiv-vmsta.cfg
\documentclass[numbers,compress,v1.0.1]{vmsta}

\volume{3}
\issue{3}
\pubyear{2016}
\firstpage{249}
\lastpage{268}
\doi{10.15559/16-VMSTA64}% Updated by VTEXPTS2LaTeX.exe, 03.11.2016
\startlocaldefs

\allowdisplaybreaks

\theoremstyle{plain}% Theorems
\newtheorem{thm}{Theorem}
\newtheorem{lemma}{Lemma}
\newtheorem{corollary}{Corollary}

\theoremstyle{definition}
\newtheorem{definition}{Definition}
\newtheorem{example}{Example}
\newtheorem{assumption}{Assumption}

\theoremstyle{remark}% Remark
\newtheorem{remark}{Remark}

\newcommand{\T}{\mathbb{T}}
\newcommand{\R}{\mathbb{R}}
\newcommand{\Sub}{\operatorname{\rm Sub}}
\newcommand{\E}{\mathbf{E}}
\renewcommand{\d}{\,\mathrm{d}}
\renewcommand{\P}{\mathbf{P}}
\endlocaldefs

\begin{document}
\begin{frontmatter}

\title{Averaged deviations of Orlicz processes and majorizing measures}
%\author[]{\inits{}\fnm{}\snm{}\corref{cor1}}\email{}
%\cortext[cor1]{Corresponding author.}
%
%\author[]{\inits{}\fnm{}\snm{}}\email{}
%
%%\fnref{f1}
%%\fntext[]{Some remarks}
%
%\address[]{}
%\address[]{}
%
%\markboth{A. Authors}{Title}

\author{\inits{R.}\fnm{Rostyslav}\snm{Yamnenko}}\email{yamnenko@univ.kiev.ua}
\address{Taras Shevchenko National University of Kyiv, Ukraine}

\markboth{R. Yamnenko}{Averaged deviations of Orlicz processes and
majorizing measures}

%\begin{abstract}
%\end{abstract}

%\begin{keyword} . \sep.
%\MSC[2010] . \sep.
%\end{keyword}

\begin{abstract}
This paper is devoted to investigation of supremum of averaged deviations
$| X(t) - f(t) - \int_{\T} (X(u) - f(u)) \d\mu(u) / \mu(\T)|$ of a
stochastic process
from Orlicz space of random variables using the method of majorizing measures.
An estimate of distribution of supremum of deviations $|X(t) - f(t)|$
is derived. A special case of the $L_q$ space is considered.
As an example, the obtained results are applied to stochastic processes
from the $L_2$ space with known covariance functions.
\end{abstract}


\received{7 September 2016}% Updated by VTEXPTS2LaTeX.exe, 03.11.2016
%10:38
%
\revised{28 October 2016}% Updated by VTEXPTS2LaTeX.exe, 03.11.2016
%10:38
%
\accepted{28 October 2016}% Updated by VTEXPTS2LaTeX.exe, 03.11.2016
%10:38
\publishedonline{11 November 2016}

\end{frontmatter}

\section{Introduction}

This paper is devoted to investigation of the supremum of averaged
deviations of stochastic processes
from Orlicz spaces of random variables using the method of majorizing
measures. In particular, we estimate functionals of the following type:
\[
\sup_{t\in\T}\left| X(t) - f(t) - \frac{1}{\mu(\T)} \int_{\T} \bigl(X(u) -
f(u)\bigr) \d\mu(u) \right|
\]
where $(\T, \mathcal{B}, \mu)$ is a measurable space with finite
measure $\mu(\T)<\infty$, and $f(u)$ is some function.
In particular, using the obtained with probability one estimates for
such a functional, we are able to estimate the distribution of
$\sup_{t\in\T} |X(t) - f(t)|$.
A~special attention is devoted to the Orlicz spaces such as the $L_q$ spaces.

The method of majorizing measures is used in the theory of Gaussian
stochastic processes to determine
conditions of boundedness and sample path continuity with probability
one of these processes. Application of the method
gives a possibility to obtain estimates for the distributions of
stochastic processes. Papers by Fernique
\cite{fernique_1971,fernique_1975} are among the first in this direction.
In some cases, the method of majorizing measures turns out to be more
effective than the entropy method exploited by Dudley \cite
{dudley_1973}, Fernique \cite{fernique_1975}, Nanopoulos and Nobelis
\cite{nanopoulos_nobelis1978}, and K\^ono \cite{kono1980}.
For example, Talagrand \cite{talagrand_1987} proposed necessary and
sufficient conditions in terms of majorizing measures for the sample
path continuity with probability one of Gaussian stochastic processes.
Such conditions in entropy terms
were found by Fernique \cite{fernique_1975} for stationary Gaussian
processes only. More details on the method of majorizing measures
can be found in papers by Talagrand \cite{talagrand_1987,talagrand_1996}, Ledoux and Talagrand \cite{ledoux_talagrand_1991}, and
Ledoux \cite{ledoux_1996}.

Particular cases of problems considered in this paper were investigated
by Koza\-chenko and Mok\-lya\-chuk \cite{kozachenko_moklyachuk2003},
Kozachenko and Ryazan\-tseva \cite{kozachenko_ryazantseva1992},
Koza\-chenko, Vasylyk, and Yamnenko \cite{kvy_ROSE}, Kozachenko and
Sergiienko \cite{kozachenko_sergiienko2014}, Yamnenko \cite{yamnenko_LMJ_2015}.
Kozachen\-ko and Ryazantseva \cite{kozachenko_ryazantseva1992} obtained
conditions of boundedness and sample path continuity
with probability one of stochastic processes from the Orlicz space of
random variables generated by exponential Orlicz functions.
Kozachenko, Vasylyk, and Yamnenko \cite{kvy_ROSE} estimated the
probability that the supremum of a stochastic process
from Orlicz spaces of exponential type exceeds some function.
Kozachenko and Moklyachuk \cite{kozachenko_moklyachuk2003} obtained
conditions of boundedness
and estimates of the distribution of the supremum of stochastic
processes from the Orlicz space of random variables.
Kozachenko and Sergiienko \cite{kozachenko_sergiienko2014} constructed
tests for a hypothesis concerning the form of the covariance function
of a Gaussian stochastic process.
Yamnenko \cite{yamnenko_LMJ_2015} obtained an estimate for
distributions of norms of deviations of a stochastic process
from the Orlicz space of exponential type from a given function in
$L_p(\T)$.

As a simple example, we apply the obtained results to a stochastic
process with the same covariance function as that of the
Ornstein--Uhlenbeck process but with trajectories from the $L_2$ space.
In \cite{yamnenko_ou_2006}, a similar problem is considered for a
generalized Ornstein--Uhlenbeck process from the Orlicz space of
exponential type $\Sub_{\varphi}(\varOmega)$.

\section{Orlicz spaces. Basic definitions}

\begin{definition}[Orlicz $N$-function \cite{bk}]
A continuous even convex function $\{U(x), x \in\R\}$ is said to be an
{Orlicz $N$-function} if it is strictly increasing for $x>0$,
$U(0)=0$, and
\[
{\frac{U(x)}{x}}\to0 \quad\mbox{ as } x\to0 \quad
\mbox{and}\quad{\frac{U(x)}{x}}\to\infty\quad\mbox{ as }
x\to\infty.
\]
\end{definition}

Any Orlicz $N$-function $U$ has the following properties \cite{kr}:
\begin{enumerate}
\item[a)] $U(\alpha x) \le\alpha U(x)$ for any $0\le\alpha\le1$;
\item[b)] $U(x) + U(y) \le U(|x| + |y|)$;
\item[c)] the function $U(x)/x$ increases for $x > 0$.\vadjust{\eject}
\end{enumerate}

\begin{example}
The following functions are $N$-functions:
\begin{itemize}
\item$U(x) = \alpha|x|^\beta, \; \alpha>0,\; \beta>1$;
\item$U(x) = \exp\{|x| \} - |x| -1$;
\item$U(x) = \exp\{\alpha|x|^\beta\} -1, \; \alpha>0,\; \beta>1$;
\item$U(x) =\left\{%
\begin{array}{ll}
( {e \alpha/ 2} )^{2 /
\alpha} x^2, & |x|\le
({2 / \alpha})^{1 / \alpha},\\
\exp\{|x|^\alpha\}, & |x| >
({2/\alpha})^{1/\alpha},\quad 0< \alpha<1.\\
\end{array}
\right.$
\end{itemize}
\end{example}

\begin{definition}[Class $\varDelta_2$ \cite{kr}]
\label{class_delta_2}
An $N$-function $U(x)$ belongs to the class $\varDelta_2$ if there exist a
constant $x_0\ge0$ and
an increasing function $K(x)>0$, $x\ge0$, such that
\[
U(zx) \le K(z) U(x)\quad\textrm{for}\ z\ge1,\ x\ge x_0.
\]
\end{definition}

\begin{example}
The following functions are from the class $\varDelta_2$:
\begin{itemize}
\item$U(x) = |x|^\alpha/ \alpha, \; \alpha>1$;
\item$U(x) = |x|^{\alpha}(|\ln|x|| + 1), \; \alpha>1$;
\item$U(x) = (1 + |x|)(|\ln(1+|x|) + 1) - |x|$.
\end{itemize}
The function $U(x) = \exp\{|x|\} - |x| - 1$ increases faster than any
power function, and therefore it does not belong to the class $\varDelta_2$.
\end{example}

%\begin{definition}[Class $\Delta^2$ \cite{kr}]
%\label{class_delta^2}
%A N-function $U(x)$ belongs to the class $\Delta^2$ if there exists
%constants $x_0\ge0$ and
%$L>1$ such that
%\[
%U^2(x) \le U(L x)\quad\textrm{for}\quad x\ge x_0.
%\]
%\end{definition}

%\begin{example}
%For any $\alpha\ge1$ and $c>0$ the function $U(x) = \exp\{c|x|^\alpha
%\} - 1$ belongs to the class
%$\Delta^2$ with $x_0 = 0$ and $L=2^{1/\alpha}$.
%\end{example}

\begin{definition}[Class $E$ \cite{bk,kozachenko1985}]
\label{def:class_E}
An $N$-function $U(x)$ belongs to the class $E$ if there exist
constants $z_0\ge0$, $B>0$, and $D>0$ such that, for all
$x\ge z_0$ and $y\ge z_0$,
\[
U(x)U(y) \le B U(D xy).
\]
\end{definition}

\begin{example}
\begin{enumerate}
\item[(i)] Let $U(x) = c|x|^p$, $c>0$, $p>1$. Then $U$
belongs to the class $E$ with constants $B=c$, $z_0 = 0$, and $D=1$.
\item[(ii)] The function $U(x) = |x|^\beta/(\log(c+|x|))^\alpha$
belongs to the class $E$ if $c$
is a number large enough such that the function $U(x)$ be convex. In
this case, $z_0 = \max\{0, \exp\{2^{-1/{\alpha}}\} - c\}$.
\end{enumerate}
\end{example}

%\begin{lemma}[\cite{bk}]
%Each function from the class $\Delta^2$ also belongs to the class $E$.
%\end{lemma}

We will further also consider functions that belong to the intersection
of the classes $\varDelta_2$ and $E$.

\begin{example}
Let $U(x) = |x|^q$, $q>1$. Then $U\in\varDelta_2 \cap E$.
\end{example}

\begin{example}
There exist functions from the class $E$ that do not belong to the
class $\varDelta_2$, for example,
$U(x) = \exp\{|x|^{\alpha}\}-1$, $\alpha> 1$, and $U(x) = \exp\{\phi
(x)\} - 1$, where $\phi(x)$ is an $N$-function.
\end{example}

Let $(\T, \mathcal{B}, \mu)$ be a measurable space with finite measure
$\mu(\T)$.

\begin{definition}[Orlicz space]
The space
$L_U^{\mu}(\T)$ of measurable functions on $(\T, \mathcal{B}, \mu)$
such that, for every
$f\in L_U^{\mu}(\T)$, there exists a constant $r_f$ such that
\[
\int_{\T} U\biggl(\frac{f(t)}{r_f}\biggr)\d\mu(t) < \infty
\]
is called the Orlicz space.
\end{definition}

The space $L_U^{\mu}(\T)$ is a Banach space with the Luxembourg norm
\begin{equation}
\label{Luxemburg_norm}
\|f\|_{U,\mu}^{\T} = \inf\biggl\{r>0\colon\int_T U\biggl(\frac
{f(t)}{r}\biggr)\d\mu(t) \le1 \biggr\}.
\end{equation}

We will also consider the Orlicz space $L_U^{\mu\times\mu}(\T\times\T)$
of measurable functions on $(\T\times\T, \mathcal{B}\times\mathcal{B},
\mu\times\mu)$, where $\mathcal{B}\times\mathcal{B}$ is the
tensor-product sigma-algebra on the product space, and
$\mu\times\mu$ is the product measure on the measurable space $(\T\times
\T, \mathcal{B}\times\mathcal{B})$, that is, for every
$f\in L_U^{\mu\times\mu}(\T\times\T)$, there exists a constant $r_f$
such that
\[
\int_{\T}\int_{\T} U\biggl(\frac{f(t,s)}{r_f}\biggr) \d\bigl(\mu(t)\times\mu
(s)\bigr) < \infty.
\]

\begin{definition}[Young--Fenchel transform]
Let $\{U(x), x\in\R\}$ be an Orlicz $N$-function. The function $\{
U^*(x), x\in\R\}$ for which
\[
U^*(x) = \sup_{y\in\R}\bigl(xy - U(y)\bigr)
\]
is called the Young--Fenchel transform of the function $U$.
\end{definition}
\begin{remark}
If $x>0$, then
\[
U^*(x) = \sup_{y>0}\bigl(xy - U(y)\bigr), \qquad U^*(-x) = U^*(x).
\]
\end{remark}

\begin{thm}[Fenchel--Moreau \cite{bk}]
Suppose that U is an N-function. Then
\[
(U^*)^* = U.
\]
\end{thm}

Let us give two examples of convex conjugate functions.
\begin{example}
\begin{enumerate}
\item[(i)] Suppose that $p>1$ and $q$ is the conjugate exponent of $p$:
$1/p + 1/q = 1$. Let $U(x) = |x|^p / p$. Then
$U^*(x) = |x^q| / q$.
\item[(ii)] Assume that $U(x) = e^{|x|} - |x| - 1$, $x\in\R$. Then
\[
U^*(x) = \bigl(1 + |x|\bigr)\bigl(\xch{\ln(1+|x|)}{|\ln(1+|x|)} + 1\bigr) - |x|, \quad x\in\R.
\]
\end{enumerate}
\end{example}

Let $U$ be an $N$-function, and $f$ be a function from the space
$L_U^{\mu}(\T)$. Consider
\[
s(f; U) = \int_\T U\bigl(f(t)\bigr)\d\mu(t) < \infty.
\]
In the space $L_U^{\mu}(\T)$, we can introduce a different norm
equivalent to the Luxembourg norm. This is the Orlicz norm
\begin{equation}
\label{Orlicz_norm}
\|f\|_{(U),\mu}^{\T} = \sup_{v\colon s(v; U^*)\le1}\left|\int_{\T}f(t)
\d\mu(t) \right|,
\end{equation}
where $U^*$ is the Young--Fenchel transform of the function $U$.

\begin{lemma}[H\"{o}lder inequality \cite{kr}]
\label{lemma:Holder}
Let $\{f(t), t\in\T\}$ be a function from the space $L^{\mu}_{U}(\T)$
endowed with the Luxembourg norm \eqref{Luxemburg_norm}, and
let $\{\varphi(t), t\in\T\}$ be a function from the space $L^{\mu
}_{(U^*)}(\T)$ endowed with the Orlicz norm \eqref{Orlicz_norm}.
Then
\begin{equation}
\int_{\T} \big|f(t)\varphi(t)\big|\d\mu(t) \le
\|f\|^{\T}_{U,\mu}\times\|\varphi\|^{\T}_{(U^*),\mu}.
\end{equation}
\end{lemma}

\begin{lemma}[Krasnoselskii and Rutitskii \cite{kr}]
\label{lemma:K&R}
Let $U(x)$ be an N-function, let $U^*(x)$ be the Young--Fenchel
transform of $U(x)$, and let $\chi_A(t)$ be the indicator function of a
set $A \subset\mathcal{B}$. Then
\begin{equation}
\|\chi_A \|^{\T}_{(U^*),\mu} = \mu(A) U^{(-1)}\biggl(\frac{1}{\mu(A)}
\biggr).
\end{equation}
\end{lemma}

Let $(\varOmega, \mathcal{F}, \P)$ be a standard probability space.
\begin{definition}
The space $L_U^{\P}(\varOmega) = L_U(\varOmega)$ of random variables $\xi= \{
\xi(\omega), \omega\in\varOmega\}$
is called an Orlicz space of random variables, that is, the Orlicz
space $L_U(\varOmega)$ is the family of random variables $\xi$
for which that there exists a constant $r_\xi>0$ such that
\[
\E U\biggl(\frac{\xi}{r_\xi}\biggr) < \infty.
\]
\end{definition}

The Luxembourg norm in this space is denoted by $\|\xi\|_U$, that is,
\[
\|\xi\|_U = \inf\biggl\{r>0\colon\E U\biggr(\frac{\xi}{r}\biggl) \le1
\biggr\}.
\]

\begin{example}
Suppose that $U(x)=|x|^p, x\in\R, p\ge1$. Then $L_U(\varOmega)$ is the
space $L_p(\varOmega)$, and the Luxembourg norm $\|\xi\|_U$ coincides with
the norm $\|\xi\|_p =\break (\E|\xi|^p)^{1/p}$.
\end{example}

The following lemma follows from the Chebyshev inequality.

\begin{lemma}[Buldygin and Kozachenko \cite{bk}]
\label{lemma:Orlicz}
Let $\xi$ be a random variable from $L_U(\varOmega)$. Then, for all $x>0$,
\begin{equation}
\P\big\{|\xi| > x\big\} \le\left(U\left(\frac{x}{\|\xi\|_U}\right)\right)^{-1}.
\end{equation}
\end{lemma}

\begin{definition}
Let $\{X(t), t\in\T\}$ be a random process. The process $X$ belongs to
the Orlicz space $L_U(\varOmega)$
if all random variables $X(t), t\in\T$, belong to the space $L_U(\varOmega
)$ and $\sup_{t\in\T}\|X(t)\|_U<\infty$.
\end{definition}

\begin{example}
Suppose that there exists a nonnegative function $c(t), t\in\T$, such
that $\P\{|X(t)|\le c(t)\} = 1$, $t\in\T$.
Then $X$ is an $L_U(\varOmega)$-process for any Orlicz space $L_U(\varOmega)$.
\end{example}

%////////////////////////////////////////////////////////////////

\section{Distribution of deviations of stochastic processes from Orlicz spaces}
Let $(\T, \rho)$ be a compact separable metric space equipped with the
metric $\rho$, and let $\mathcal{B}$ be the Borel $\sigma$-algebra on
$(\T, \rho)$.

%\begin{definition}
%A process $X$ belongs to the Orlicz space $L_U^{\mu}(\T)$ with
%probability one if
%all sample paths of $X$ are measurable with respect to $\mathcal{B}$
%and with probability one
%belong to $L_U^{\mu}(\T)$.
%\end{definition}

Consider a separable stochastic \xch{process}{process process} $X = \{X(t), t\in\T\}$
from the Orlicz space $L_U(\varOmega)$, that is, $X(t)\in L_U(\varOmega)$,
$t\in\T$, is continuous in the norm $\|\cdot\|_U$.

\begin{assumption}\label{assumption_sigma}
Consider such a function $\sigma= \{\sigma(h), h > 0\}$, $t\in\T$,
such that
\begin{itemize}
\item$\sigma(h)\ge0$,
\item$\sigma(h)$ increases in $h>0$,
\item$\sigma(h)\to0$ as $h\to0$,
\item$\sigma(h)$ is continuous, and
\item$\sup_{\rho(t,s)\le h} \| X(t) - X(s) \|_U \le\sigma(h)$.
\end{itemize}
\end{assumption}

Note that at least one such function exists, for example,
\[
\sigma(h) = \sup_{\rho(t,s)\le h} \big\|X(t) - X(s)\big\|_U.
\]

Denote by $\sigma^{(-1)}(h)$ the generalized inverse to $\sigma(h)$,
that is, $\sigma^{(-1)}(h) =\break \sup\{s\colon\sigma(s) \le h\}$.
Put
\[
d(u,v) = \big\| X(u) - X(v)\big\|_U
\]
and
\[
d_f(u,v) = \big\| X(u) - X(v) - f(u) + f(v)\big\|_U
\]
and let $S$ be a set from $\mathcal{B}$ such that
\begin{equation}
\label{S}
(\mu\times\mu)\big\{(u,v)\in S\times S\colon\rho(u,v)\neq0\big\} > 0.
\end{equation}
Consider a sequence $\epsilon_k(t) > 0$ such that $\epsilon_k(t) >
\epsilon_{k+1}(t)$, $\epsilon_k(t) \to0$ as $k \to\infty$,
and $\epsilon_1(t) = \sup_{s\in S}\rho(t,s)$.
%\begin{notations}
Put $C_t(u) = \{s\colon\rho(t,s)\le u\}$, $C_{t,k} = C_t(\epsilon_k(t))$,
$\mu_k(t) = \mu(C_{t,k}\cap S)$.
%\end{notations}

\begin{assumption} \label{assumption:f}
Assume that, for a continuous function $f=\{f(t), t\in\T\}$, there
exists a continuous increasing function $\delta(y) > 0$, $y>0$, such that
$\delta(y)\to0$ as $y\to0$
and the following condition is satisfied:
\[
\big|f(u)-f(v)\big| \le\delta\bigl(\big\| X(u) - X(v)\big\|_U\bigr) \le d(u,v).
\]
\end{assumption}

Throughout the paper, we will assume that, for all $B\in\mathcal{B}$,
\[
\int_B \big|X(u) - f(u)\big| \d\mu(u) < \infty.
\]

\begin{lemma}\label{lemma:1}
Suppose that $X = \{X(t), t\in\T\}$ is a separable stochastic process
from the Orlicz space $L_U(\varOmega)$ that satisfies Assumption \ref
{assumption_sigma}. Let $f$ be a function satisfying Assumption~\ref
{assumption:f}, let $\zeta(y), y>0$, be an arbitrary continuous
increasing function such that $\zeta(y)\to0$ as $y\to0$, and let
\[
\frac{X(u) - X(v) - f(u) + f(v)}{\zeta(d_f(u,v))} \in L_U^{\mu\times\mu
}(\T\times\T).
\]

Then, for any $S\in\mathcal{B}$ satisfying \eqref{S}, we have the
following inequality with probability one:
\begin{align}\label{lemma:1_main}
&\sup_{t\in S}\left| X(t) - f(t) - \frac{1}{\mu(S)} \int_S \bigl(X(u) - f(u)\bigr)
\d\mu(u) \right| \nonumber\\
&\quad{\le}\,\left\| \frac{X(u) \,{-}\, X(v) \,{-}\, f(u) \,{+}\, f(v)}{\zeta(d_f(u,v))}\right\|
_{U, \mu\times\mu}^{S\times S}\sup_{t\in S}\sum_{l=1}^{\infty} \zeta\bigl(2\sigma\bigl(\epsilon_l(t)\bigr)\bigr)
U^{(-1)}\biggl(\frac{1}{\mu^2_{l+1}(t)}\biggr).
\end{align}
\end{lemma}

\begin{proof}
Let $V$ be the set of separability of the process $X$, and let $t$ be
an~arbitrary point from $S\cap V$. Put
\[
\tau_l(u) = \frac{\chi_{C_{t,l}\cap S}(u)}{\mu_l(t)},
\]
where $\chi_A(u)$ is the indicator function of $A$. Then
\begin{align}
&\left\| X(t) - f (t) - \int_S \bigl(X(u) - f(u)\bigr)\tau_l(u) \d\mu(u) \right\|
_U \nonumber\\
&\quad\le\int_S \big\| \bigl(X(t) - X(u)\bigr) - \bigl(f(t)-f(u)\bigr) \big\|_U \tau_l(u) \d\mu(u)
\nonumber\\
&\quad\le\int_S \big\| X(t) - X(u) \big\|_U \tau_l(u) \d\mu(u) + \int_S \big| f(t)-f(u)\big|
\tau_l(u) \d\mu(u) \nonumber\\
&\quad\le\sigma\bigl(\epsilon_l(t)\bigr) + \delta\bigl(\sigma(\epsilon_l(t))\bigr) \to0
\label{lmm1:eq1}
\end{align}
as $l\to\infty$. If follows from Lemma~\ref{lemma:Orlicz} and (\ref
{lmm1:eq1}) that
\[
\int_S \bigl(X(u) - f(u)\bigr)\tau_l(u) \d\mu(u) \to X(t) - f (t)
\]
in probability as $l\to\infty$. Therefore, there exists a sequence
$l_n$ such that
\[
\int_S \bigl(X(u) - f(u)\bigr)\tau_{l_n}(u) \d\mu(u) \to X(t) - f (t)
\]
with probability one as $l_n\to\infty$. It is easy to see that
\begin{align}
&\left| X(t) - f (t) - \int_S \bigl(X(u) - f(u)\bigr)\tau_l(u) \d\mu(u) \right
|\qquad\nonumber\\
&\quad= \bigg| X(t) - f (t) - \int_S \bigl(X(u) - f(u)\bigr)\tau_{l_n}(u) \d\mu(u)
\nonumber\\
&\qquad+ \sum_{l=1}^{l_n-1}\left( \int_S \bigl(X(u) - f(u)\bigr)\tau_{l+1}(u) \d\mu(u)
- \int_S \bigl(X(u) - f(u)\bigr)\tau_{l}(u) \d\mu(u)\right) \bigg| \nonumber\\
&\quad\le\bigg| X(t) - f (t) - \int_S \bigl(X(u) - f(u)\bigr)\tau_{l_n}(u) \d\mu
(u)\bigg| \nonumber\\
&\qquad+ \sum_{l=1}^{l_n-1}\bigg| \int_S \bigl(X(u) - f(u)\bigr)\tau_{l+1}(u) \d\mu(u)
- \int_S \bigl(X(u) - f(u)\bigr)\tau_{l}(u) \d\mu(u) \bigg|.
\label{lmm1:eq2}
\end{align}
It follows from (\ref{lmm1:eq2}) that the following inequality holds
with probability one:
\begin{align}
&\left| X(t) - f (t) - \frac{1}{\mu(S)}\int_S \bigl(X(u) - f(u)\bigr)\d\mu(u)
\right| \nonumber\\
&\quad\le\sum_{l=1}^{\infty}\bigg| \int_S \bigl(X(u) - f(u)\bigr)\tau_{l+1}(u) \d\mu
(u) - \int_S \bigl(X(u) - f(u)\bigr)\tau_{l}(u) \d\mu(u) \bigg|\nonumber\\
&\quad= \sum_{l=1}^{\infty}\bigg| \int_S\int_S \bigl(X(u) - X(v) - f(u) + f(v)\bigr)\tau
_{l+1}(u)\tau_{l}(v) \d\mu(u) \d\mu(v) \bigg|\nonumber\\
&\quad\le \int_{S\times S}\bigg|\frac{X(u) - X(v) - f(u) + f(v)}{\zeta
(d_f(u,v))}\bigg|\nonumber\\
&\qquad\times\Biggl(\sum_{l=1}^{\infty}\tau_{l+1}(u)\tau_{l}(v)\zeta
(d_f(u,v))\Biggr) \d(\mu(u) \times\mu(v)).
\label{lmm1:eq3}
\end{align}

From Lemma~\ref{lemma:Holder} and (\ref{lmm1:eq3}) we have the inequality
\begin{align}
&\left| X(t) - f (t) - \frac{1}{\mu(S)}\int_S \bigl(X(u) - f(u)\bigr)\d\mu(u)
\right| \nonumber\\
&\quad{\le}\, \bigg\|\frac{X(u) \,{-}\, X(v) \,{-}\, f(u) \,{+}\, f(v)}{\zeta(d_f(u,v))}\bigg\|
_{U,\mu\times\mu}^{S\times S}\left\|\sum_{l=1}^{\infty}\tau_{l+1}(u)\tau_{l}(v)\zeta
\bigl(d_f(u,v)\bigr)\right\|_{(U^*),\mu\times\mu}^{S\times S}.
\label{lmm1:eq4}
\end{align}

Also, we have
\begin{align}
\tau_{l+1}(u)\tau_l(u)\zeta\bigl(d_f(u,v)\bigr)&\le\tau_{l+1}(u)\tau_l(u)\zeta
\bigl(d_f(u,t) + d_f(u,t)\bigr) \nonumber\\
&\le\tau_{l+1}(u)\tau_l(u)\zeta\bigl(\sigma\bigl(\epsilon_l(t)\bigr) + \sigma\bigl(\epsilon
_{l+1}(t)\bigr)\bigr)\notag\\
&\le\tau_{l+1}(u)\tau_l(u)\zeta\bigl(2\sigma\bigl(\epsilon_l(t)\bigr)\bigr).
\label{lmm1:eq5}
\end{align}

From (\ref{lmm1:eq4}) and (\ref{lmm1:eq5}) we have that with
probability one the following inequality holds:
\begin{align}
&\left| X(t) - f (t) - \frac{1}{\mu(S)}\int_S \bigl(X(u) - f(u)\bigr)\d\mu(u)
\right| \nonumber\\
&\quad\le \bigg\|\frac{X(u) - X(v) - f(u) + f(v)}{\zeta(d_f(u,v))}\bigg\|
_{U,\mu\times\mu}^{S\times S} \nonumber\\
&\qquad\times\sum_{l=1}^{\infty}\zeta\bigl(2\sigma\bigl(\epsilon_l(t)\bigr)\bigr) \big\|\tau
_{l+1}(u)\tau_{l}(v)\big\|_{(U^*),\mu\times\mu}^{S\times S}.
\label{lmm1:eq6}
\end{align}

It follows from Lemma~\ref{lemma:K&R} that
\begin{align}
\big\|\tau_{l+1}(u)\tau_{l}(v)\big\|_{(U^*),\mu\times\mu}^{S\times S}
&=\frac{1}{\mu_l(t)\mu_{l+1}(t)}\left\|\chi_{C_{t,l}\cap S}(u)\chi
_{C_{t,l+1}\cap S}(v)\right\|_{(U^*),\mu\times\mu}^{S\times S}
\nonumber\\
&=U^{(-1)}\left(\frac{1}{\mu_l(t)\mu_{l+1}(t)} \right) \le U^{(-1)}
\biggl(\frac{1}{\mu_{l+1}^2(t)} \biggr).
\label{lmm1:eq7}
\end{align}

Since $t\in S\cap V$ and $S\cap V$ is a countable set, (\ref{lmm1:eq7})
holds with probability one for all
$t\in S\cap V$. The process $X$ is separable, and therefore
\begin{align*}
&\sup_{t\in S} \left| X(t) - f (t) - \frac{1}{\mu(S)}\int_S \bigl(X(u) -
f(u)\bigr)\d\mu(u) \right|\\[6pt]
&\quad= \sup_{t\in S\cap V} \left| X(t) - f (t) - \frac{1}{\mu(S)}\int_S
\bigl(X(u) - f(u)\bigr)\d\mu(u) \right|
\end{align*}
with probability one.
\end{proof}
\begin{remark}
If the right side of (\ref{lemma:1_main}) is finite, then the measure
$\mu$ is called a majorizing measure on $S$ for the process $X$.
\end{remark}

\begin{corollary}\label{corollary:3.1}
Let the assumptions of Lemma~\ref{lemma:1} be satisfied. Put
\[
\zeta_1(t) = \zeta\bigl(2\sigma\bigl(\epsilon_1(t)\bigr)\bigr) = \zeta\Bigl(2\sigma
\Bigl(\sup_{s\in S}\rho(t,s)\Bigr)\Bigr)
\]
and
\[
\nu_t(u) = \mu\bigl(C_t\bigl(\sigma^{(-1)}\bigl(\zeta^{(-1)}(u)/2\bigr)
\bigr)\cap S \bigr).
\]
Then, for any $0<p<1$, we have the inequality
\begin{equation}
\sup_{t\in S}\left| X(t) - f(t) - \frac{1}{\mu(S)} \int_S \bigl(X(u) - f(u)\bigr)
\d\mu(u) \right|
\le\eta_f C_p
\label{eq:cor1}
\end{equation}
with probability one, where
\begin{equation}
\label{eta_f}
\eta_f = \left\|\frac{X(u) - X(v) - f(u) + f(v)}{\zeta(d_f(u,v))} \right
\|_{U,\mu\times\mu}^{S\times S}
\end{equation}
and
\begin{equation}
\label{C_p}
C_p = \sup_{t\in S} \frac{1}{p(1-p)}\int_0^{p\zeta_1(t)} U^{(-1)}
\bigl(\bigl(\nu_t(u)\bigr)^{-2} \bigr) \d u.
\end{equation}

\end{corollary}

\begin{proof}
Let the sequence $\epsilon_k(t), k\ge1$, be defined as
\[
\epsilon_k(t) = \sigma^{(-1)}\bigl(\zeta^{(-1)}\bigl(\zeta_1(t)p^{k-1}\bigr)\bigr).
\]
Then
\[
\zeta\bigl(2\sigma\bigl(\epsilon_l(t)\bigr)\bigr) = \zeta_1(t) p^{l-1}
\]
and
\[
\mu_{l+1}(t) = \mu\bigl(C_t\bigl(\epsilon_{l+1}(t)\bigr) \cap S\bigr) = \nu_t\bigl(\zeta_1(t)p^l\bigr).
\]
Therefore, from (\ref{lemma:1_main}) and the following inequality we
obtain the assertion of the corollary:\vadjust{\eject}
\begin{align*}
\!\sum_{l=1}^{\infty} \zeta\bigl(2\sigma\bigl(\epsilon_l(t)\bigr)\bigr)
U^{(-1)}\!\biggl(\frac{1}{\mu^2_{l+1}(t)}\biggr) &\,{=}\, \sum_{l=1}^{\infty}\zeta
_1(t)p^{l-1} U^{(-1)}
\bigl( \bigl(\nu_t\bigl(\zeta_1(t)p^l\bigr)\bigr)^{-2} \bigr)
\\
&\,{\le}\,\sum_{l\ge1}\frac{1}{p(1-p)}\int_{\zeta_1(t)p^{l+1}}^{\zeta
_1(t)p^{l}} U^{(-1)}\bigl(\nu_t(u)^{-2}\bigr) \d u\\
&\,{\le}\,\int_{0}^{\zeta_1(t)p} U^{(-1)}\bigl(\nu_t(u)^{-2}\bigr) \d u .
\qedhere
\end{align*}
\end{proof}

\begin{remark}
We will further find additional conditions on $\eta_f$ and $C_p$ from
(\ref{eta_f}) and (\ref{C_p}) such that
the constant $C_p$ is finite and the random variable $\eta_f$ is finite
with probability one.
In this case, we get that $\mu$ is a majorizing measure on $S$ for $X$.
In Theorems \ref{theorem:2} and \ref{theorem_3}, these conditions will
be formulated for processes from the class $\varDelta_2$ and space
$L_q(\varOmega)$.
\end{remark}

\begin{thm}\label{theorem:1}
Let assumptions of Lemma~\ref{lemma:1} be satisfied, and let the
following conditions hold:
\begin{enumerate}
\item[a)]\ \vspace*{-18pt}
\[
\sup_{t\in S} \int_0^{\zeta_1(t)} U^{(-1)} \bigl( \bigl(\nu_t(u)\bigr)^{-2}
\bigr) \d u < \infty,
\]
\item[b)] there exists a constant $r>0$ such that
\begin{equation}
\int_S \int_S \E U \left( \frac{|X(u) - X(v)| + |f(u) - f(v)|}{\zeta
(d_f(u,v)) r} \right) \d\bigl(\mu(u)\times\mu(v)\bigr) < \infty.
\label{eq:thm1_ass_1}
\end{equation}
\end{enumerate}
Then, for all $x>0$, we have the inequality
\begin{align}
&\P\Bigl\{ \sup_{t\in S}\big|X(t) - f(t)\big| > x \Bigr\} \nonumber\\
&\quad\le\inf_{0\le\alpha\le1} \inf_{0<p<1}
\biggl[ \left(U\left({\alpha x}\Big/{\left\| \int_S \bigl(X(u) - f(u)\bigr) \frac{\d\mu
(u)}{\mu(S)}\right\|_U } \right) \right)^{-1}\nonumber\\
&\qquad+\P\biggl\{ \eta_f > \frac{(1-\alpha) x}{C_p} \biggr\}
\biggr],
\end{align}
where $\eta_f$ and $C_p$ are defined in (\ref{eta_f}) and (\ref{C_p}),
respectively.
\end{thm}
\begin{proof}
Using Fubini's theorem and (\ref{eq:thm1_ass_1}), we obtain that with
probability one
\begin{align}
&\int_S\int_S U\left(\frac{X(u) - X(v) - f(u) + f(v)}{\zeta(d_f(u,v))
r}\right)\d\bigl(\mu(u)\times\mu(v)\bigr)
\nonumber\\
&\quad\le\int_S\int_S U\left(\frac{|X(u) - X(v)| + |f(u) - f(v)|}{\zeta
(d_f(u,v)) r}\right)\d\bigl(\mu(u)\times\mu(v)\bigr) < \infty, \nonumber
\end{align}
that is, the process
\[
\frac{X(u) - X(v)-f(u)+f(v)}{\zeta(d_f(u,v))}
\]
with probability one belongs to the space $L_U^{\mu\times\mu}(S\times
S)$. Therefore, with probability one
\[
\eta_f = \left\|\frac{X(u) - X(v) - f(u) + f(v)}{\zeta(d_f(u,v))} \right
\|_{U,\mu\times\mu}^{S\times S}\vadjust{\eject}
\]
is a finite random variable. It follows from (\ref{eq:cor1}) that
\begin{equation}
\sup_{t\in S}\big|X(t) - f(t)\big| \le\frac{1}{\mu(S)}\left|\int_S \bigl(X(u) -
f(u)\bigr) \d\mu(u) \right| + \eta_f C_p
\label{eq:thm1_1}
\end{equation}
with probability one. Since $X(u) \in L_U(\varOmega)$ for $u\in S$, we
have $X(u) - f(u) \in L_U(\varOmega)$ for $u\in S$ and
\[
\frac{1}{\mu(S)} \int_S \bigl(X(u) - f(u)\bigr)\d\mu(u) \in L_U(\varOmega).
\]
Moreover,
\begin{align*}
\left\|\frac{1}{\mu(S)} \int_S \bigl(X(u) - f(u)\bigr) \d\mu(u)\right\|_U
&\le\frac{1}{\mu(S)} \int_S \big\| X(u) - f(u) \big\|_U \d\mu(u)\\
&\le\sup_{u\in S}\big\| X(u) - f(u) \big\|_U < \infty.
\end{align*}

It follows from Lemma~\ref{lemma:Orlicz} that, for any $y>0$,
\begin{align}
\P\left\{ \left|\int_S \bigl(X(u) - f(u)\bigr) \frac{\d\mu(u)}{\mu(S)} \right| >
y \right\}
\le1 / U\left(\frac{y}{\|\frac{1}{\mu(S)} \int_S (X(u) - f(u)) \d
\mu(u)\|_U}\right).
\label{eq:thm1_2}
\end{align}

It follows from (\ref{eq:thm1_1}) that, for any $0\le\alpha\le1$ and $x
> 0$,
\begin{align}
&\P\left\{ \sup_{t\in S}\big |X(t) - f(t)\big| > x \right\} \nonumber\\
&\quad\le\P\left\{ \frac{1}{\mu(S)} \left|\int_S \bigl(X(u) - f(u)\bigr) \d\mu(u)
\right| > \alpha x \right\}
+\P\big\{\eta_f C_p > (1-\alpha) x\big\}.
\label{eq:thm1_3}
\end{align}
The statement of the theorem follows from (\ref{eq:thm1_2}) and (\ref
{eq:thm1_3}).
\end{proof}

%////////////////////////////////////////////////////////////////

\section{Distribution of deviations of stochastic processes from
classes $\varDelta_2$ and $\varDelta_2 \cap E$}

\begin{definition}
A stochastic process $X=\{X(t), t\in\T\}$ belongs to the class $\varDelta
_2$ if $X\in L_U(\varOmega)$,
where $U$ is an Orlicz function from the class $\varDelta_2$.
\end{definition}

\begin{thm}
\label{theorem:2}
Suppose that $X = \{X(t), t\in\T\}$ is a separable stochastic process
from the class $\varDelta_2$
that satisfies Assumption \ref{assumption_sigma}. Let $f$ be a function
satisfying Assumption~\ref{assumption:f},
where $U$ is the Orlicz N-function from the class $\varDelta_2$, let $\zeta
(y), y>0$, be an arbitrary continuous increasing function such that
$\zeta(y)\to0$ as $y\to0$, and let
\[
\frac{X(u) - X(v) - f(u) + f(v)}{\zeta(d_f(u,v))} \in L_U^{\mu\times\mu
}(\T\times\T).
\]
Suppose that the following conditions are satisfied:
\begin{itemize}
\item[a)] there exists a constant $r>0$ such that
\begin{equation}
\int_S\int_S K \biggl(\frac{\gamma(d_f(u,v))}{r} \biggr) \d\bigl(\mu
(u)\times\mu(v)\bigr) < \infty,
\label{eq:theorem:2_1}
\end{equation}
where $K$ and $x_0$ are introduced in Definition \ref{class_delta_2} of
the class $\varDelta_2$ and $\gamma(u) = u/\zeta(u)$;
\item[b)]\ \vspace*{-19pt}
\begin{equation}
\sup_{t\in S}\int_0^{\zeta_1(t)} U^{(-1)}\bigl(\bigl(\nu_t(u)\bigr)^{-2}\bigr) \d u <
\infty,
\label{eq:theorem:2_2}
\end{equation}
where $\zeta_1(t)$ and $\nu_t(u)$ are defined in Corollary~\ref{corollary:3.1}.
\end{itemize}
Then, for any $0<p<1$, the following inequality holds with probability one:
\begin{align}
&\sup_{t\in S}\left|X(t) - f(t) - \int_S \bigl(X(u) - f(u)\bigr) \frac{\d\mu
(u)}{\mu(S)} \right| \nonumber\\
&\quad\le\frac{\eta_f}{p(1-p)} \sup_{t\in S}\int_0^{\zeta_1(t) p}
U^{(-1)}\bigl(\bigl(\nu_t(u)\bigr)^{-2}\bigr) \d u,
\label{eq:theorem:2_3}
\end{align}
where
\[
\eta_f = \left\|\frac{X(u) - X(v) - f(u) + f(v)}{\zeta(d_f(u,v))} \right
\|_{U, \mu\times\mu}^{S\times S}
\]
is a finite with probability one random variable.
\end{thm}

\begin{proof}
It is easy to see that the assumptions of Lemma~\ref{lemma:1} are satisfied.
Consider the function $\eta_f$. In order to show that it is finite with
probability one,
it suffices to prove that the random function
\[
\frac{(X(u) - X(v) -f(u) + f(v))}{\zeta(d_f(u,v))}
\]
belongs to the space $L_{U}^{\mu\times\mu}(S\times S)$ with probability
one. For this, it suffices to show
that there exists a number $r>0$ such that
\[
\int_S\int_S U\biggl(\frac{X(u) - X(v) - f(u) + f(v)}{r\zeta(d_f(u,v))}
\biggr) \d\bigl(\mu(u)\times\mu(v)\bigr)<\infty
\]
with probability one. It follows from Fubini's theorem that it suffices
to prove that
\begin{equation}\label{eq:theorem:2_4}
\int_S\int_S \E U\biggl(\frac{(X(u) - X(v) - f(u) + f(v))}{r\zeta
(d_f(u,v))} \biggr) \d\bigl(\mu(u)\times\mu(v)\bigr)<\infty.
\end{equation}

Since $U\in\varDelta_2$, using Assumption~\ref{assumption:f}, we have
\begin{align}
&\E U\biggl(\frac{(X(u) - X(v) - f(u) + f(v))}{r\zeta(d_f(u,v))} \biggr)
\nonumber\\
&\quad= \E\chi_{\frac{|X(u) - X(v) - f(u) + f(v)|}{d_f(u,v)} > x_0} \chi
_{\frac{\gamma(d_f(u,v))}{r} > 1} \nonumber\\
&\qquad\times U\biggl(\frac{X(u) - X(v) - f(u) + f(v)}{d_f(u,v)} \biggr) K
\biggl(\frac{\gamma(d_f(u,v))}{r} \biggr) \nonumber\\
&\qquad+ \E\chi_{\frac{|X(u) - X(v) - f(u) + f(v)|}{d_f(u,v)} \le x_0} \chi
_{\frac{\gamma(d_f(u,v))}{r} > 1}
U\biggl(x_0 \frac{\gamma(d_f(u,v))}{r} \biggr) \nonumber\\
&\qquad+ \chi_{\frac{\gamma(d_f(u,v))}{r} \le1} \E U\biggl(\frac{X(u) - X(v)
- f(u) + f(v)}{d_f(u,v)} \biggr) \nonumber\\
&\quad\le\E U\biggl(\frac{X(u) - X(v) - f(u) + f(v)}{d_f(u,v)} \biggr) K
\biggl(\frac{\gamma(d_f(u,v))}{r} \biggr) \nonumber\\
&\qquad+ U(x_0) K\biggl(\frac{\gamma(d_f(u,v))}{r} \biggr) \nonumber\\
&\qquad+ \chi_{\frac{\gamma(d_f(u,v))}{r} \le1} \E U\biggl(\frac{X(u) - X(v)
- f(u) + f(v)}{d_f(u,v)} \biggr) \nonumber\\
&\quad\le\biggl(K \biggl(\frac{\gamma(d_f(u,v))}{r} \biggr) + \chi_{\frac
{\gamma(d_f(u,v))}{r} \le1} \biggr)
\E U\biggl(\frac{X(u) - X(v) - f(u) + f(v)}{d_f(u,v)} \biggr) \nonumber
\\
&\qquad+ U(x_0) K\biggl(\frac{\gamma(d_f(u,v))}{r} \biggr) \nonumber\\
&\quad\le K \biggl(\frac{\gamma(d_f(u,v))}{r} \biggr)\bigl(1+U(x_0)\bigr) + \chi_{\frac
{\gamma(d_f(u,v))}{r} \le1}.
\label{eq:theorem:2_5}
\end{align}
Therefore, for all $r$ such that inequality (\ref{eq:theorem:2_1})
holds, we have the relation
\begin{align}
\label{eq:theorem:KozMokl_4.1:4.6}
&\E\int_S \int_S U\biggl(\frac{X(u) - X(v) - f(u) + f(v)}{r\zeta
(d_f(u,v))} \biggr) \d\bigl(\mu(u)\times\mu(v)\bigr) \nonumber\\
&\quad\le\int_S \int_S \chi_{\frac{\gamma(d_f(u,v))}{r} \le1} \d\bigl(\mu
(u)\times\mu(v)\bigr)\nonumber\\
&\qquad+ \bigl(1 + U(x_0)\bigr) \int_S \int_S K \biggl(\frac{\gamma
(d_f(u,v))}{r} \biggr) \d\bigl(\mu(u)\times\mu(v)\bigr) < \infty.
\end{align}
Inequality (\ref{eq:theorem:2_4}) and the statement of Theorem \ref
{theorem:2} follows from the last relation.
\end{proof}

\begin{corollary}
\label{cor:KozMokl_4.1}
Let the assumptions of Theorem~\ref{theorem:2} be satisfied. Let $r$ be
a number such that condition \eqref{eq:theorem:2_1}
holds. Then, for any $x>r$, we have the inequality
\[
\P\{\eta_f>x\} \le Z(x),
\]
where
\[
Z(x)=\int_S\int_S\biggl[\chi_{\frac{\gamma(d_f(u,v))}{x}\le
1} + \bigl(1+U(x_0)\bigr) K\biggl( \frac{\gamma(d_f(u,v))}{x}\biggr) \biggr]
\d\bigl(\mu(u)\times\mu(v)\bigr).
\]
\end{corollary}
\begin{proof}
It follows from (\ref{eq:theorem:2_3}) and Chebyshev's inequality that
\begin{align}
&\P\{\eta_f > x\} \nonumber\\
&\quad= \P\left\{\int_S \int_S U\biggl(\frac{X(u) - X(v) - f(u) + f(v)}{x \zeta
(d_f(u,v))} \biggr)
\d\bigl(\mu(u)\times\mu(v)\bigr) > 1\right\} \nonumber\\
&\quad\le \E\int_S \int_S U\biggl(\frac{X(u) - X(v) - f(u) +
f(v)}{x\zeta(d_f(u,v))} \biggr) \d\bigl(\mu(u)\times\mu(v)\bigr)
\nonumber\\
&\quad\le Z(x).\\[-20pt]\notag
\end{align}
\end{proof}

\begin{corollary}\label{corollary_4.2}
Let the assumptions of Theorem~\ref{theorem:2} be satisfied. Let
$U(x)\in\varDelta_2\cap E$ and $z_0=0$ in Definition~\ref{def:class_E}.
Then, for any $x>0$, we have the inequality
\[
\P\{\eta_f>x\} \le\frac{Z(r)B}{U(x/Dr)},
\]
where B and D are the constants from Definition~\ref{def:class_E}, and
$r$ is a constant such that condition (\ref{eq:theorem:2_1})
holds, $Z(r)$ is defined in Corollary~\ref{cor:KozMokl_4.1}, and
\[
Z(r)\le\mu^2(S)+(1+U(x_0))\int_S\int_SK\biggl( \frac
{\gamma(d_f(u,v))}{r}\biggr) \d\bigl(\mu(u)\times\mu(v)\bigr) = Z_1(r).
\]
\end{corollary}
\begin{proof}
It follows from (\ref{eq:theorem:KozMokl_4.1:4.6}), the definition of
class $E$, and Chebyshev's inequality that
\begin{align}
&\P\{\eta_f > x\} \nonumber\\
&\quad= \P\biggl\{\int_S \int_S U\biggl(\frac{X(u) - X(v) - f(u) + f(v)}{x\zeta
(d_f(u,v))} \biggr) \d\bigl(\mu(u)\times\mu(v)\bigr) > 1\biggr\} \nonumber\\
&\quad\le\E\int_S \int_S U\biggl(\frac{X(u) - X(v) - f(u) +
f(v)}{d_f(u,v)} \frac{\gamma(d_f(u,v))}{x} \biggr) \d\bigl(\mu(u)\times\mu
(v)\bigr) \nonumber\\
&\quad= \frac{1}{U(\frac{x}{Dr})} \E\int_S \int_S U\biggl(\frac
{X(u) - X(v) - f(u) + f(v)}{d_f(u,v)} \frac{\gamma(d_f(u,v))}{x}
\biggr) \nonumber\\
&\qquad\times U\biggl(\frac{x}{D r}\biggr)\d\bigl(\mu(u)\times\mu(v)\bigr)\notag\\
&\quad\le\frac{B}{U(x/(D r))} \nonumber\\
&\qquad\times\E\int_S \int_S U\biggl(\frac{X(u) - X(v) - f(u) + f(v)}{d_f(u,v)}
\frac{\gamma(d_f(u,v))}{r} \biggr) \d\bigl(\mu(u)\times\mu(v)\bigr) \nonumber\\
&\quad\le\frac{Z(r) B}{U(x/(Dr))}.\\[-26pt]\notag
\end{align}
\end{proof}

\begin{corollary}\label{corollary:4.2n}
Let the assumptions of Theorem~\ref{theorem:2} be satisfied. Then

a) for all $x>r$, we have the inequality
\begin{align}\label{eq:corollary:4.2n:1}
&\P\Bigl\{\sup_{t\in S} \big|X(t) - f(t)\big| > x \Bigr\} \nonumber\\
&\quad\le\inf_{0<\alpha<1}\inf_{0<p<1}
\biggl(1/ U\biggl(\frac{x\alpha}{\|\int_S (X(u) \,{-}\, f(u))\d\mu(u) / \mu(S)
\|_U} \biggr) {+}\, Z\biggl(\frac{x(1\,{-}\,\alpha)}{C_p} \biggr)\!\biggr),
\end{align}
where $Z(x)$ is determined in Corollary~\ref{cor:KozMokl_4.1}, $C_p$ is
determined in Theorem \ref{theorem:1}, and $r$ is a constant such that
condition \eqref{eq:theorem:2_1}
holds;\vadjust{\eject}

b) if $U\in\varDelta_2 \cap E$ with $z_0 = 0$, then, for all $x>0$, we
have the inequality
\begin{align}
&\P\Bigl\{\sup_{t\in S} \big|X(t) - f(t)\big| > x \Bigr\} \nonumber\\
&\quad\le\inf_{0<\alpha<1}\inf_{0<p<1} \biggl(1/ U\biggl(\frac{x\alpha}{\|\int
_S (X(u) - f(u)) \d\mu(u) / \mu(S) \|_U} \biggr) \nonumber\\
&\qquad + Z(r)B/U\biggl(\frac{x(1-\alpha)}{D r C_p} \biggr)\biggr),
\end{align}
where $B$ and $D$ are the constants determined in Definition~\ref
{def:class_E}, $r$ is a constant such that condition (\ref
{eq:theorem:2_1}) holds true
and $Z(x)$ is determined in Corollary~\ref{cor:KozMokl_4.1}.
\end{corollary}
\begin{proof}
Statement a) follows from Theorem \ref{theorem:1} and Corollary~\ref
{cor:KozMokl_4.1}. Statement b) follows from Theorem~\ref{theorem:1} and
Corollary~\ref{corollary_4.2}.
\end{proof}

\begin{thm}\label{theorem_3}
Suppose that $X = \{X(t), t\in\T\}$ is a separable stochastic process
from the space $L_q(\varOmega)$, $q>1$, satisfying Assumption \ref
{assumption_sigma}. Let $f \in L_{q}^{\mu}(S)$ be a~function satisfying
Assumption~\ref{assumption:f}, let $\zeta(y), y>0$, be an arbitrary
continuous increasing function such that $\zeta(y)\to0$ as $y\to0$,
and let the
following conditions hold:
\begin{align*}
&\varDelta_q = \int_S\int_S \left(\gamma\bigl(d_f(u,v)\bigr)\right)^q \d\bigl(\mu(u)\times
\mu(v)\bigr) < \infty,\\
&\sup_{t\in S}\int_0^{\zeta_1(t)} \bigl(\nu_t(u)\bigr)^{-2/q} \d u < \infty,
\end{align*}
where $\gamma(y) = y/\zeta(y)$, $\zeta_1(t)$ and $\nu_t(u)$ are defined
in Corollary~\ref{corollary:3.1}. Then, for any $0<p<1$ and $x>0$, we
have the inequality
\begin{equation}\label{eq:theorem_3:3}
\P\Bigl\{\sup_{t\in S} \big|X(t) - f(t)\big| > x \Bigr\}
\le x^{-q} \bigl(\varGamma_q^{\frac{1}{q+1}} +
\bigl(D_{p,q}^q \varDelta_q \bigr)^{\frac{1}{q+1}} \bigr)^{q+1},
\end{equation}
where
\begin{align}
\varGamma_q &= \E\left(\int_S \bigl(X(u) - f(u)\bigr) \frac{\d\mu(u)}{\mu(S)} \right)^q,\notag\\
\label{eq:theorem_3:Dpq}
D_{p,q}&=\sup_{t\in S}\frac{1}{p(1-p)}\int_0^{\zeta_1(t) p} \bigl(\nu
_t(u)\bigr)^{-2/q} \d u.
\end{align}
\end{thm}

\begin{proof}
Consider inequality (\ref{eq:corollary:4.2n:1}). In this case,
\[
\left\|\int_S \bigl(X(u) - f(u)\bigr) \frac{\d\mu(u)}{\mu(S)} \right\|_U =
\varGamma_q^{1/q},
\]
$B=D=1$, $x_0 = 0$, $K(y) = y^q$, $r>0$,
\[
C_p = \sup_{t\in S}\frac{1}{p(1-p)} \int_0^{p\zeta_1(t)} \bigl(\nu
_t(u)\bigr)^{-\frac{2}{q}} \d u,
\]
and $Z(r) r^{q} \to\varDelta_q$ as $r\to0$, where\vadjust{\eject}
\begin{align*}
Z(r) r^{q} &= r^q \int_S\int_S\chi_{\frac{\gamma(d_f(u,v))}{r}\le1} \d
\bigl(\mu(u)\times\mu(v)\bigr)\\
&\quad+ \int_S\int_S \bigl(\gamma\bigl(d_f(u,v)\bigr)\bigr)^q \d\bigl(\mu(u)\times\mu(v)\bigr).
\end{align*}
It follows from (\ref{eq:corollary:4.2n:1}) that, for any $0<p<1$,
\[
\P\Bigl\{\sup_{t\in S}\big|X(t) - f(t)\big| > x \Bigr\} \le\inf_{0\le\alpha
\le1}
\biggl(\frac{\varGamma_q}{\alpha^q x^q} + \frac{C_{p}^q \varDelta_q}{(1-\alpha
)^q x^{q}}\biggr).
\]
Inequality (\ref{eq:theorem_3:3}) follows from the last inequality
after taking the infimum with respect to $\alpha$.
\end{proof}

\section{Example of existence of majorizing measure for $L_2(\varOmega)$-process}

In this section, we show that the Lebesgue measure is majorizing on $S$
for some process $X$ from the space $L_2(\varOmega)$.

%\begin{remark}
%All stationary random processes from the space $L_q(\varOmega)$ with $
%\sigma(h) \le h^H$, $H\in(0,1)$ satisfy Theorem \ref{theorem_3}. An
%%example of such process from $L_2(\varOmega)$ is given below in Theorem
%\ref{theorem_4}.
%\end{remark}

Let $S = \T= [0, T]$. Assume that $\rho(u,v) = d_f(u,v) = |u-v|$ and
let $\mu$ be the Lebesgue measure,
that is, $\mu(S) = T$. Then
\[
C_t(u) = \big\{s \colon|t - s| \le u \big\} = [t-u, t+u]
\]
and
\[
C_t \cap S = \min\{T, t+u\} - \max\{0, t-u\}.
\]
The function $\zeta(u) = u^\alpha$,
$\alpha> 0$, satisfies the condition of Lemma \ref{lemma:1};
therefore, $\gamma(u) = u^{1-\alpha}$ and the expressions in Theorem
\ref{theorem_3} take the following form:
\begin{align}
\nu_t(u) &= \min\left\{T, t+\sigma^{(-1)}\left(\frac{1}{2}u^{1/\alpha
}\right)\right\} -
\max\left\{0, t-\sigma^{(-1)}\left(\frac{1}{2}u^{1/\alpha}\right)\right
\},
\label{ex:nu_t(u)}\\
\zeta_1(t) &= \zeta\Bigl(2\sigma\Bigl(\sup_{s\in S}|t-s|\Bigr)\Bigr)
= \bigl(2 \sigma\bigl(\max\{t, T-t\}\bigr)\bigr)^{\alpha},\notag
\end{align}
and
\[
\varDelta_q = \int_0^T \int_0^T \bigl(d_f(u,v)\bigr)^{(1 -\alpha) q} \d u \d v.
\]

Let $q = 2$, that is, $X(t)$ is a stochastic process from $L_2(\varOmega
)$. Assume that $X$ is a centered process with
covariance function $R_X(u, v) = \E X(u) X(v)$. Then using Fubini's
theorem, we obtain the following representation of
$\varGamma_q$ from Theorem \ref{theorem_3}:
\begin{align*}
\varGamma_q &= \E\biggl(\int_0^T \bigl(X(u) - f(u)\bigr) \frac{\d u}{T}\biggr)^2\\
&=\frac{1}{T^2} \int_0^T \int_0^T \E\bigl(X(u) - f(u)\bigr)\bigl(X(v) - f(v)\bigr) \d v \d u\\
&=\frac{1}{T^2}\int_0^T\int_0^T R_X(u, v)\d u \d v + \frac{1}{T^2}
\biggl(\int_0^T f(v) \d v\biggr)^2.
\end{align*}

Consider the following stochastic process.

\begin{definition}
A stochastic process $X = \{X(t), t \in\T\}$ is called the
generalized Ornstein--Uhlenbeck process from the space $L_2(\varOmega)$ if
$X$ is an $L_2(\varOmega)$-process with\vadjust{\eject} the covariance function
\[
R_X(t, s) = e^{-\tau|t-s|}, \quad\tau> 0.
\]
\end{definition}

Then from Theorem~\ref{theorem_3} we can state conditions for a
majorizing measure on $[0,T]$ for the process $X$.

\begin{thm}\label{theorem_4}
Let $X=\{X(t), t\in[0, T]\}$ be a centered separable
generalized\break Ornstein--Uhlenbeck stochastic process from the space
$L_2(\varOmega)$ satisfying Assumption~\ref{assumption_sigma}, and let a
function $f$ satisfy Assumption~\ref{assumption:f} with the function
$\delta(t)$, $t>0$, such that
\begin{align}
\int_0^T\int_0^T \bigl(\delta\bigl((u - v)^{\beta_1/2}\bigr)\bigr)^{2-2\alpha}
\d u \d v < \infty,
\end{align}
where $\alpha\in(2/\beta_2, 1/\beta_1 +1)$ with $\beta_1,\beta_2\in
(0,1)$ such that $2/\beta_2 < 1/\beta_1 +1$.
Then the Lebesgue measure is majorizing on $[0, T]$ for the process
$X$, and, for any $0<p<1$ and $x>0$, we have the inequality
\begin{equation}\label{eq:theorem_4:3}
\P\Bigl\{\sup_{t\in[0, T]} \big|X(t) - f(t)\big| > x \Bigr\} \le x^{-2}
\Bigl(\varGamma_2^{\frac{1}{3}} +
\inf_{\alpha\in(0,1)}(D_{p,2}^2 \varDelta_2)^{\frac{1}{3}} \Bigr)^{3},
\end{equation}
where
\begin{align*}
\varGamma_2 &= \frac{2(T \tau+ e^{-\tau T} - 1)}{\tau^2 T^2} + \frac
{1}{T^2}\left(\int_0^T f(v) \d v\right)^2,\\
\varDelta_2 &= \int_0^T\int_0^T \bigl(2\bigl(\tau|u-v|\bigr)^{\beta_1} + \bigl(\delta
\bigl(\bigl(2\tau(u - v)\bigr)^{\beta_1/2}\bigr)\bigr)^2\bigr)^{1-\alpha} \d u \d v,\\
D_{p,2}&=
\frac{1}{p(1-p)}\sup_{t\in[0,T]}\biggl(
\frac{2\tau' (\tau'\min\{t, T-t\})^{\alpha\beta_2/2-1}}{1-2/\alpha}
\\
&\quad
+\frac{p2^{3\alpha/2}\left(\tau\max\{t, T-t\}\right)^{\alpha\beta_2/2}
- (\tau'\min\{t, T-t\})^{\alpha\beta_2/2}}{T}\biggr),
\end{align*}
where $\tau' = \tau2^{3/\beta_2}$.
\end{thm}

\begin{proof}
Let us apply the inequality
\begin{equation}\label{exp_ineq}
1-\exp\{-x\} \le x^\beta,\quad0<\beta\le1,\; x \ge0.
\end{equation}
It is easy to see that, for all $0\le x < 1$, we have $1-\exp\{-x\} \le
x \le x^{\beta}$. Also, $1-\exp\{-x\} \le1 \le x^{\beta}$ for all
$x\ge1$.

Then, using (\ref{exp_ineq}), we have that
\begin{align*}
d(t,s) &= \big\|X(t) - X(s)\big\|_{L_2} = \bigl(\E\bigl(X(t) - X(s)\bigr)^2\bigr)^{1/2}\\
&= \bigl(\E X(t)^2 + \E X(s)^2 - 2R_X(t,s)\bigr)^{1/2} = \bigl(2-2\exp\big\{{-}\tau
|t-s| \big\}\bigr)^{1/2}\\
&\le2^{1/2} \bigl(\tau|t-s|\bigr)^{\beta/2},
\end{align*}
that is, the function $\sigma(h) = 2^{1/2}(\tau h)^{\beta/2}\ge\sup
_{|t-s|\le h} d(t,s)$, $h > 0$, satisfies Assumption~\ref{assumption_sigma}.
Then\vadjust{\eject}
\begin{equation}
\sigma^{(-1)}(h) = \frac{h^{2/\beta}}{2^{1/\beta}\tau},\quad h>0.
\label{ex:sigma_1}
\end{equation}

Also, it is easy to see that, for the centered process $X$,
\[
d_f(t,s) = \bigl(d^2(t,s) + \bigl(f(t) - f(s)\bigr)^2\bigr)^{1/2} \le\bigl(2\bigl(\tau
|t-s|\bigr)^{\beta_1} + \delta^2\bigl(d(t,s)\bigr)\bigr)^{1/2}
\]
for any $\beta_1\in(0,1]$ and
\begin{align}
\int_0^T\int_0^T R_X(s, t)\d s \d t& =
\int_0^T\int_0^t e^{-\tau(t-s)}\d s \d t + \int_0^T\int_t^T e^{-\tau
(s-t)} \d s \d t \nonumber\\
&= \frac{1}{\tau}\int_0^T\bigl( 1- e^{-\tau t} - e^{-\tau(T-t)} + 1 \bigr) \d t
\,{=}\, \frac{2(T \tau\,{+}\, e^{-\tau T} \,{-}\, 1)}{\tau^2}.
\end{align}

From (\ref{eq:theorem_3:Dpq}) it follows that
\[
\varDelta_2 = \int_0^T\int_0^T \bigl(2\bigl(\tau|t-s|\bigr)^{\beta_1} + \bigl(\delta
\bigl(\bigl(2\tau(u - v)\bigr)^{1/2}\bigr)\bigr)^2\bigr)^{1-\alpha} \d u \d v < \infty
\]
if $\beta_1(1-\alpha)+1 > 0$, that is, if $\alpha< 1/\beta_1 +1$. Then
\[
\int_0^T\int_0^T \bigl(\bigl(\delta(u - v)\bigr)^{\beta_1/2}\bigr)^{2-2\alpha}
\d u \d v < \infty.
\]

Applying (\ref{ex:sigma_1}) to (\ref{ex:nu_t(u)}) for some $\beta_2\in
(0,1]$, we have that
\[
\nu_t(u) =
\min\biggl\{T, t+\frac{u^{\frac{2}{\alpha\beta_2}}}{\tau2^{3/\beta
_2}}\biggr\} -
\max\biggl\{0, t-\frac{u^{\frac{2}{\alpha\beta_2}}}{\tau2^{3/\beta
_2}}\biggr\}.
\]
Put $\tau' = \tau2^{3/\beta_2}$.
It is easy to see that $\nu_t(u) = T$ if $T < t+\frac{u^{\frac{2}{\alpha
\beta_2}}}{\tau'}$
and $0 > t-\frac{u^{\frac{2}{\alpha\beta_2}}}{\tau'}$,
that is, if $u > (\tau'\max\{t, T-t\})^{\alpha\beta_2/2}$;
$\nu_t(u) = t + \frac{u^{\frac{2}{\alpha\beta_2}}}{\tau'} - (t-\frac
{u^{\frac{2}{\alpha\beta_2}}}{\tau'}) =
\frac{u^{\frac{2}{\alpha\beta_2}}}{2\tau'}$ if $u \le(\tau'\min\{t,
T-t\})^{\alpha\beta_2/2}$;
and $\nu_t(u) = \max\{t, T-t\} + \frac{u^{\frac{2}{\alpha\beta_2}}}{\tau
'}$ if
$(\tau'\min\{t, T-t\})^{\alpha\beta_2/2} \le u < (\tau' \max\{t, T-t\}
)^{\alpha\beta_2/2}$.

Consider
\[
D_{p,2}=\sup_{t\in[0,T]}\frac{1}{p(1-p)}\int_0^{p2^{3\alpha/2}\left
(\tau\max\{t, T-t\}\right)^{\alpha\beta_2/2}}
\frac{1}{\nu_t(u)} \d u.
\]
For $\alpha> 2 / \beta_2$, we have
\begin{align}
&\int_0^{p2^{3\alpha/2}\left(\tau\max\{t, T-t\}\right)^{\alpha\beta_2/2}}
\frac{\d u}{\nu_t(u)}\notag\\
&\quad= \int_0^{(4\tau\min\{t, T-t\})^{\alpha/2}} 2\tau' u^{-2/(\alpha\beta
_2)} \d u\nonumber\\
&\qquad+ \int_{(\tau'\min\{t, T-t\})^{\alpha\beta_2/2}}^{(\tau'\max\{t, T-t\}
)^{\alpha\beta_2/2}} \frac{\d u}{\max\{t, T-t\}
+ \frac{u^{2/(\alpha\beta_2)}}{\tau'}} \nonumber\\
&\qquad+\int_{(\tau'\max\{t, T-t\})^{\alpha\beta_2/2}}^{p2^{3\alpha/2}\left
(\tau\max\{t, T-t\}\right)^{\alpha\beta_2/2}} \frac{1}{T} \d u
\nonumber\\
&\quad\le\frac{2\tau'}{1-2/(\alpha\beta_2)} \bigl(\tau'\min\{t, T-t\}\bigr)^{\alpha
\beta_2/2-1} \nonumber\\
&\qquad+\frac{(\tau'\max\{t, T-t\})^{\alpha\beta_2/2} - (\tau'\min\{t, T-t\}
)^{\alpha\beta_2/2}}{\max\{t, T-t\} + \min\{t, T-t\}} \nonumber\\
&\qquad+\frac{p2^{3\alpha/2}\left(\tau\max\{t, T-t\}\right)^{\alpha\beta_2/2}
- (\tau'\max\{t, T-t\})^{\alpha\beta_2/2}}{T}\nonumber\\
&\quad= \frac{2\tau' (\tau'\min\{t, T-t\})^{\alpha\beta_2/2-1}}{1-2/\alpha}
\nonumber\\
&\qquad+\frac{p2^{3\alpha/2}\left(\tau\max\{t, T-t\}\right)^{\alpha\beta_2/2}
- (\tau'\min\{t, T-t\})^{\alpha\beta_2/2}}{T}.\\[-28pt]\notag
\end{align}
\end{proof}

%% Acknowledgements %%
%%%%%%%%%%%%%%%%%%%%%%
\section*{Acknowledgments}
The author gratefully thanks Prof. Dr. Yu. Kozachenko, who provided
insight and expertise for the research.

\end{document}